\documentclass{amsart}
\usepackage{setspace}
\usepackage{mathrsfs}
\usepackage{amssymb}
\usepackage{latexsym}
\usepackage{amsfonts}
\usepackage{amsmath}
\usepackage{eucal}
\usepackage{bm}
\usepackage{bbm}
\usepackage{graphicx}
\usepackage[english]{varioref}
\usepackage[nice]{nicefrac}
\usepackage[all]{xy}
\usepackage{amsthm}
\usepackage{amssymb,amsthm,upref,amscd}

\def\op{\operatorname}

\def\mmod{\kern-1pt\operatorname{-mod}}

\newtheorem{Thm}{Theorem}[section]
\newtheorem{Lem}[Thm]{Lemma}
\newtheorem{Cor}[Thm]{Corollary}
\newtheorem{Prop}[Thm]{Proposition}

\newtheorem{Rem}[Thm]{Remark}

\numberwithin{equation}{section}

\begin{document}

\title[Alvis-Curtis Duality]{Alvis-Curtis Duality for Representations of  Reductive Groups with Frobenius Maps}

\author{Junbin Dong}
\address{Institute of Mathematical Sciences, ShanghaiTech University, Shanghai 201210, P.R. China.}
\email{dongjunbin1990@126.com}



\date{April 13, 2020}


\begin{abstract}
We generalize the Alvis-Curtis duality to the abstract representations of reductive groups with Frobenius maps. Similar to  the case of representations of finite reductive groups, we show that the Alvis-Curtis duality of infinite type which we define in this paper also interchanges the irreducible representations in the principal representation category.
\end{abstract}

\maketitle

\section{Introduction}
Let ${\bf G}$ be a connected reductive group defined over the finite field $\mathbb{F}_q$. Assume that $\Bbbk$ is a field which may  be different from $\bar{\mathbb{F}}_q$. According to the results of Borel and Tits (see \cite{BT}), except the trivial representation, all other irreducible representations of  ${\bf G}$ are infinite-dimensional when  ${\bf G}$  is semisimple and $\op{char}\Bbbk\neq \op{char} \bar{\mathbb{F}}_q$.
So it seems difficult to study the abstract representations of ${\bf G}$ over an arbitrary field $\Bbbk$.
Now let $G_{q^a}$ be the set of $\mathbb{F}_{q^a}$-points of $\bf G$, then ${\bf G}=\bigcup G_{q^a}$. Around 2013,  N,Xi provided a new and fundamental method to study the abstract representations of ${\bf G}$ over ${\bf \Bbbk}$ by taking the direct limit of the finite-dimensional representations of $G_{q^a}$ and got many interesting results (see \cite{X}). For example, given an irreducible $G_{q^a}$-module $M_{q^a}$ for each positive integer $a$ which satisfies $M_{q^a}\subseteq M_{q^b}$ as $\Bbbk$-vector spaces when $a\mid b$,  one has that $M=\bigcup M_{q^a}$ is a simple $\Bbbk {\bf G}$-module. However there are counterexamples to show that not all simple $\Bbbk {\bf G}$-modules can be obtained in this way (see Section 4 in \cite{CD1}).

One important class of irreducible modules of a reductive group (resp. Lie algebra) comes from certain induced modules from an one-dimensional character of a Borel subgroup (resp. Borel subalgebra).
For the rational representations of algebraic groups and, the representations of Lie algebras in the BGG category $\mathcal{O}$, it was known that all irreducible modules are the simple quotients of Weyl modules and Verma modules, respectively. Now let ${\bf B}$ be a Borel subgroup of ${\bf G}$ and let $\bf T$ be a maximal torus contained in $\bf B$, both being defined over $\mathbb{F}_q$.
Motivated by these classical results, we consider the abstract induced module $\mathbb{M}(\theta)=\Bbbk{\bf G}\otimes_{\Bbbk{\bf B}}{\bf \Bbbk}_\theta$ in \cite{CD3}, where $\theta$ is a character of $\bf T$ which can also be regarded as a $\Bbbk {\bf B}$-module ${\bf \Bbbk}_\theta$.
All the composition factors of $\mathbb{M}(\theta)$ are given when $\op{char}\Bbbk\neq \op{char} \bar{\mathbb{F}}_q$ or $\Bbbk=\bar{\mathbb{F}}_q$ and $\theta$ is rational. So we have got  a series of infinite dimensional simple $\Bbbk {\bf G}$-modules. However it seems to be impossible to obtain all simple $\Bbbk {\bf G}$-modules.

The Alvis-Curtis duality is a duality operation on the characters of finite reductive groups, introduced and studied by D. Alvis (see \cite{Al}) and  C.W. Curtis (see \cite{Cu}). P. Deligne and G. Lusztig consider this duality on the representations (see \cite{DL1},  \cite{DL2}). The Alvis-Curtis duality interchanges some irreducible representations, for example, the trivial and Steinberg representations. Thus in this paper we try to generalize this duality to the abstract representations of a reductive group ${\bf G}$.
Similar to the finite-type case, it also interchanges simple $\Bbbk {\bf G}$-modules in the principal representation category  $\mathscr{O}({\bf G})$ (see Theorem \ref{dual of basis}), for example, the trivial module and the infinite dimensional Steinberg module.

\medskip
\noindent{\bf Acknowledgements.} The author is grateful to Prof. Nanhua Xi for his constant encouragement and guidance. The author  thanks Prof. Ming Fang and Dr. Xiaoyu Chen for their useful discussions. The author would also like to thank the referees for careful reading and  helpful comments.

%
%

\section{Category $\mathscr{O}({\bf G})$ and three bases of $K_0(\mathscr{O}({\bf G}))$ }

In this section we recall some results in \cite{CD3} about the composition factors of the abstract induced module $\mathbb{M}(\theta)=\Bbbk{\bf G}\otimes_{\Bbbk{\bf B}}{\bf \Bbbk}_\theta$. Then we introduce the principal representation category $\mathscr{O}({\bf G})$ and give three bases of $K_0(\mathscr{O}({\bf G}))$. Firstly we introduce some notations for convenience. The setting of ${\bf G} ,{\bf B},{\bf T} $ is as before. We denote by $\Phi=\Phi({\bf G};{\bf T})$ the corresponding root system, and by $\Phi^+$ (resp. $\Phi^-$) the set of positive (resp. negative) roots determined by ${\bf B}$. Let $W=N_{\bf G}({\bf T})/{\bf T}$ be the corresponding Weyl group. We denote by $\Delta=\{\alpha_i\mid i\in I\}$ the set of simple roots and by $S=\{s_i\mid i\in I\}$ the corresponding simple reflections in $W$. For each $w\in W$, let $\dot{w}$ be one representative in $N_{\bf G}({\bf T})$. We denote by ${\bf U}=R_u({\bf B})$  the ($F$-stable) unipotent radical of ${\bf B}$ and for any $w\in W$, let ${\bf U}_w$ (resp. ${\bf U}_w'$) be the subgroup of ${\bf U}$ generated by all ${\bf U}_\alpha$ (the root subgroup of $\alpha\in\Phi^+$) with $w\alpha\in\Phi^-$ (resp. $w\alpha\in\Phi^+$). The multiplication map ${\bf U}_w\times{\bf U}_w'\rightarrow{\bf U}$ is a bijection (see \cite{Ca}).
For any subgroup $\bf H$ of $\bf G$ defined over $\mathbb{F}_q$ and
any power $q^a$ of $q$, denote by $H_{q^a}$ the set of
$\mathbb{F}_{q^a}$-points of $\bf H$. Then we have ${\bf H}=\bigcup_{a=1}^{\infty}H_{q^a}$.

\medskip
From now on  we assume that $\Bbbk$ is an algebraically closed field which satisfies $\Bbbk=\bar{\mathbb{F}}_q$ if  $\op{char}\Bbbk =\op{char} \bar{\mathbb{F}}_q$.
All representations we consider are over  $\Bbbk$. Let $\widehat{\bf T}$ be the set of characters of ${\bf T}$ when  $\op{char}\Bbbk\neq \op{char} \bar{\mathbb{F}}_q$ and let it be the set of rational characters when $\Bbbk=\bar{\mathbb{F}}_q$. Then the Weyl group $W$ acts naturally on $\widehat{\bf T}$ by
$$(w\cdot \theta ) (t):=\theta^w(t)=\theta(\dot{w}^{-1}t\dot{w})$$
for any $\theta\in\widehat{\bf T}$. We write $x{\bf 1}_{\theta}:=x\otimes{\bf 1}_{\theta}\in\mathbb{M}(\theta)$,  where ${\bf 1}_{\theta}$ is a nonzero element in ${\bf \Bbbk}_\theta$.
Then it is easy to see  that $\mathbb{M}(\theta)=\sum_{w\in W}\Bbbk {\bf U}_{w^{-1}}\dot{w}{\bf 1}_{\theta}$ by the Bruhat decomposition and moreover, $\{ u \dot{w}{\bf 1}_{\theta} \mid  w\in W, u\in {\bf U}_{w^{-1}}\} $ is a basis of $\mathbb{M}(\theta)$.

For each $i \in I$, let ${\bf G}_i$ be the subgroup of $\bf G$ generated by ${\bf U}_{\alpha_i}, {\bf U}_{-\alpha_i}$ and we set ${\bf T}_i= {\bf T}\cap {\bf G}_i$. For $\theta\in\widehat{\bf T}$, define the subset $I(\theta)$ of $I$ by $$I(\theta)=\{i\in I \mid \theta| _{{\bf T}_i} \ \text {is trivial}\}.$$
 For $J\subset I(\theta)$, let ${\bf G}_J$ be the subgroup of $\bf G$ generated by ${\bf G}_i$, $i\in J$. We choose a representative $\dot{w}\in {\bf G}_J$ for each $w\in W_J$. Thus the element $w{\bf 1}_\theta:=\dot{w}{\bf 1}_\theta$  $(w\in W_J)$ is well defined. For $J\subset I(\theta)$, we set
$$\eta(\theta)_J=\sum_{w\in W_J}(-1)^{\ell(w)}w{\bf 1}_{\theta},$$
and let $\Delta(\theta)_J=\displaystyle \sum_{w\in W}\Bbbk {\bf U}\dot{w}\eta(\theta)_J$,which is a submodule of $\mathbb{M}(\theta)$.
For any $J\subset I(\theta)$, we define
$$E(\theta)_J=\Delta(\theta)_J/\Delta(\theta)_J',$$
where $\Delta(\theta)_J'$ is the sum of all $\Delta(\theta)_K$ with $J\subsetneq K\subset I(\theta)$. The following theorem (see \cite{CD3}) gives all the composition factors of $\mathbb{M}(\theta)$ explicitly.

\begin{Thm}\label{EJ}  Let $\theta \in \widehat{\bf T}$. Then all the $\Bbbk{\bf G}$-modules $E(\theta)_J$ $(J\subset I(\theta))$ are irreducible and pairwise non-isomorphic. In particular, $\mathbb{M}(\theta)$ has exactly $2^{|I(\theta)|}$ composition factors each of multiplicity one.
\end{Thm}

Now we introduce the principal representation category $\mathscr{O}({\bf G})$. It is the full subcategory of $\Bbbk{\bf G}$-Mod such that any object $M$ in $\mathscr{O}({\bf G})$
is of finite length and its composition factors are $E(\theta)_J$ for some $\theta \in \widehat{\bf T}$ and $J\subset I(\theta)$. Denote by $K_0(\mathscr{O}({\bf G}))$ the Grothendieck group of $\mathscr{O}({\bf G})$. It is obvious that $\{[E(\theta)_J] \mid \theta\in \widehat{\bf T}, J\subset I(\theta)\}$ is a basis of $K_0(\mathscr{O}({\bf G}))$.

\medskip
The irreducible $\Bbbk {\bf G}$-module $E(\theta)_J$ can also be realized by parabolic induction.
Now let ${\bf P}_J=\langle {\bf T},{\bf U}_{\alpha}\mid \alpha \in \Phi^+\cup \Phi_J \rangle$ be the standard parabolic group containing ${\bf B}$. We have the Levi decomposition ${\bf P}_J={\bf U}_J\rtimes {\bf L}_J$, where
\begin{equation*} {\bf U}_J:=\prod_{\alpha\in \Phi^+\backslash \Phi_J}{\bf U}_\alpha\qquad {\bf L}_J:=\langle {\bf T},{\bf U}_{\alpha}\mid \alpha\in \Phi_J \rangle.\end{equation*}
Let $\theta\in\widehat{\bf T}$ and $K\subset I(\theta)$. Because $\theta|_{{\bf T}_i}$ is trivial for all $i\in K$, it induces a character (still denoted by $\theta$) of $\overline{\bf T}={\bf T}/[{\bf L}_K,{\bf L}_K]\cap{\bf T}$. Therefore, $\theta$ is regarded as a character of ${\bf L}_K$ by the homomorphism ${\bf L}_K\rightarrow\overline{\bf T}$ (with the kernel $[{\bf L}_K,{\bf L}_K]$), and hence as a character of ${\bf P}_K$ by letting ${\bf U}_K$ acts trivially. Set $\mathbb{M}(\theta, K):=\Bbbk{\bf G}\otimes_{\Bbbk{\bf P}_K}\theta$.  Let ${\bf 1}_{\theta, K}$ be a nonzero element in the one-dimensional module $\Bbbk_\theta$ associated to $\theta$. We abbreviate $x{\bf 1}_{\theta, K}:=x\otimes{\bf 1}_{\theta, K} \in \mathbb{M}(\theta, K)$ as before.

Now set $\mathscr{R}(w)=\{i\in I\mid ws_i< w \}$ and $Z_K= \{ w\in W\mid \mathscr{R}(w)\subset I\backslash K \}$. Then by the same argument of \cite[Lemma 6.2]{D}, we have
$$\mathbb{M}(\theta, K)=\sum_{w\in Z_K}\Bbbk {\bf U}_{w^{-1}}\dot{w}{\bf 1}_{\theta, K} $$
and $\{ u \dot{w}{\bf 1}_{\theta, K}  \mid  w\in Z_K, u\in {\bf U}_{w^{-1}}\} $ is a basis of
$\mathbb{M}(\theta, K)$. Therefore $\mathbb{M}(\theta, K)$ is
a indecomposable $\Bbbk {\bf G}$-module. Indeed, we use the basis of $\mathbb{M}(\theta, K)$ to consider the endomorphism algebra of $\mathbb{M}(\theta, K)$ and have
$$\op{End}_{\bf G}(\mathbb{M}(\theta, K))\cong \op{Hom}_{{\bf P}_K}(\Bbbk_{\theta}, \mathbb{M}(\theta, K))\cong \Bbbk$$ which implies that  $\mathbb{M}(\theta, K)$ is indecomposable.

Using the same proof as \cite[Theorem 6.3]{D} and \cite[Corollary 3.8]{CD1}, we have the following proposition.
\begin{Prop}\label{Parabolic}
For $K \subset I(\theta)$, $\mathbb{M}(\theta, K) \cong \displaystyle \mathbb{M}(\theta)\Big{/} \sum_{s\in {K}}\Delta(\theta)_{\{s\}}$. All composition factors of $\mathbb{M}(\theta, K)$ are $E(\theta)_J$ with $J\subset I(\theta)\backslash K $.
\end{Prop}

For $J\subset I(\theta)$, set $J'=I(\theta)\backslash J$ and   $\nabla(\theta)_J= \mathbb{M}(\theta, J')=\Bbbk{\bf G}\otimes_{\Bbbk{\bf P}_{J'}}\Bbbk_\theta$.  Let $E(\theta)_J'$ be the submodule of $\nabla(\theta)_J$ generated by $$D(\theta)_J:=\sum_{w\in W_J}(-1)^{\ell(w)}\dot{w}{\bf 1}_{\theta, J'},$$
Thus we see that $E(\theta)_J'$ is isomorphic to $E(\theta)_J$ as $\Bbbk {\bf G}$-modules by \cite[Proposition 1.9]{CD3}. Therefore $E(\theta)_J$ can be regarded as the scole of  $\nabla(\theta)_J$.

\medskip

Given a finite set $X$, a matrix $(T_{K,L})$ indexed by subsets of $X$ will be said to be \textit{strictly upper triangular} (resp. \textit{strictly lower triangular})  if $T_{K,L}=0$ unless $K\subset L$ (resp. $K\supset L $), and \textit{strictly upper unitriangular} (resp. \textit{strictly lower unitriangular})   if also $T_{K,K}=1$ for all $K\subset X$.

For a fixed $\theta \in \widehat{\bf T}$ and $J\subset I(\theta)$, by Theorem \ref{EJ} we have
$[\Delta(\theta)_J]=\sum_{J\subset K} [E(\theta)_K] $. So for such $\theta\in \widehat{\bf T}$, the transition matrix $(A_{K,L})$ indexed by subsets of $I(\theta)$ from $\{[E(\theta)_J] \mid J\subset I(\theta)\}$ to $\{[\Delta(\theta)_J] \mid J\subset I(\theta)\}$ is strictly upper unitriangular. Explicitly, the transition matrix $(A_{K,L})$ is
$$A_{K,L}=\left\{
\begin{array}{ll}
1 &\ \mbox{if}~K\subset L,\\
0 &\ \mbox{otherwise.}
\end{array}\right.$$
Let $(B_{K,L})$ be the inverse matrix of $(A_{K,L})$. Then $(B_{K,L})$ is the transition matrix from $\{[\Delta(\theta)_J] \mid J\subset I(\theta)\}$ to $\{[E(\theta)_J] \mid J\subset I(\theta)\}$ which  is also strictly upper unitriangular. Moreover the  matrix $(B_{K,L})$ is
$$B_{K,L}=\left\{
\begin{array}{ll}
(-1)^{|L|-|K|} &\ \mbox{if}~K\subset L,\\
0 &\ \mbox{otherwise.}
\end{array}\right.$$

\smallskip
On the other hand, we have $[\nabla(\theta)_J]=\sum_{K\subset J} [E(\theta)_K] $. So the transition matrix $(C_{K,L})$ from $\{[E(\theta)_J] \mid J\subset I(\theta)\}$ to $\{[\nabla(\theta)_J] \mid J\subset I(\theta)\}$ is
$$C_{K,L}=\left\{
\begin{array}{ll}
1 &\ \mbox{if}~K\supset L,\\
0 &\ \mbox{otherwise,}
\end{array}\right.$$
which is strictly lower unitriangular. Let $(D_{K,L})$ be the inverse matrix of $(C_{K,L})$. Then
$$D_{K,L}=\left\{
\begin{array}{ll}
(-1)^{|K|-|L|} &\ \mbox{if}~K\supset L,\\
0 &\ \mbox{otherwise,}
\end{array}\right.$$
which is also strictly lower unitriangular.
Thus the following proposition is proved.
\begin{Prop}\label{two basis}
\noindent $(1)$ One has that $\{[E(\theta)_J] \mid \theta\in \widehat{\bf T}, J\subset I(\theta)\}$, $\{[\Delta(\theta)_J] \mid \theta\in \widehat{\bf T}, J\subset I(\theta)\}$ and $\{[\nabla(\theta)_J] \mid \theta\in \widehat{\bf T}, J\subset I(\theta)\}$ are three bases of $K_0(\mathscr{O}({\bf G}))$.

\noindent $(2)$ For a fixed  character $\theta\in \widehat{\bf T}$, the transition matrices between $\{[E(\theta)_J] \mid J\subset I(\theta)\}$ and $\{[\Delta(\theta)_J] \mid J\subset I(\theta)\}$ are given by $(A_{K,L})$ and $(B_{K,L})$.

\noindent $(3)$ For a fixed  character $\theta\in \widehat{\bf T}$, the transition matrices between $\{[E(\theta)_J] \mid J\subset I(\theta)\}$ and $\{[\nabla(\theta)_J] \mid J\subset I(\theta)\}$ are given by $(C_{K,L})$ and $(D_{K,L})$.
\end{Prop}

\section{Harish-Chandra Induction and Restriction}
In this section, we firstly generalize the Harish-Chandra induction and restriction to the representations of infinite reductive groups. Then we show that these two functors are well defined on the principal representation category.  As before ${\bf G}$ is a connected reductive group defined over $\mathbb{F}_q$.
Let ${\bf P}$ be a parabolic subgroup  and let ${\bf L}$ a Levi subgroup with Levi decomposition ${\bf P} = {\bf L} {\bf V}$. By convention, $\Bbbk [{\bf G}/{\bf V}]$ is the $\Bbbk$-vector space with basis ${\bf G}/{\bf V}$ on which ${\bf G}$ (resp. $\bf L$) acts by left (resp. right) translations on the basis vectors. This is well defined as $\bf L$ normalizes $\bf V$ and then $\Bbbk [{\bf G}/{\bf V}]$ is a $(\Bbbk {\bf G} , \Bbbk {\bf L})$-bimodule. Similarly, $\Bbbk [{\bf V}\backslash {\bf G}]$ with the action of ${\bf G}$ (resp. ${\bf L}$) given by right (resp. left) translations is a $(\Bbbk {\bf L}, \Bbbk {\bf G})$-bimodule.

We define the Harish-Chandra induction:
$$ \mathcal{R}^{{\bf G}}_{{\bf L}\subset {\bf P}}: \ \ \Bbbk {\bf L}\text{-Mod} \longrightarrow \Bbbk {\bf G}\text{-Mod} \ \ \ \ \  N\longrightarrow \Bbbk [{\bf G}/{\bf V}]\otimes_{\Bbbk {\bf L}} N$$
and the Harish-Chandra restriction:
$$ {}^*\mathcal{R}^{{\bf G}}_{{\bf L}\subset {\bf P}}: \ \ \Bbbk {\bf G}\text{-Mod} \longrightarrow \Bbbk {\bf L}\text{-Mod} \ \ \ \ \  M \longrightarrow \Bbbk [{\bf V}\backslash {\bf G}] \otimes_{\Bbbk {\bf G}} M.$$
This construction is formally an extension of the ordinary Harish-Chandra induction and restriction of finite groups with BN-pairs.

\begin{Prop}\label{functor} The functors $\mathcal{R}^{{\bf G}}_{{\bf L}\subset {\bf P}}$ and ${}^*\mathcal{R}^{{\bf G}}_{{\bf L}\subset {\bf P}}$  are exact.
\end{Prop}

\begin{proof}  The restriction functor ${}^*\mathcal{R}^{{\bf G}}_{{\bf L}\subset {\bf P}}$ is obviously exact. The  induction functor  $\mathcal{R}^{{\bf G}}_{{\bf L}\subset {\bf P}}$, it is the same as the functor $\Bbbk {\bf G}\otimes_{\Bbbk {\bf P}}\Bbbk [{\bf P}/{\bf V}]\otimes_{\Bbbk {\bf L}}-$. Since $\Bbbk [{\bf P}/{\bf V}]\cong \Bbbk {\bf L}$ as right $\Bbbk {\bf L}$-modules and $\Bbbk {\bf G}$ is a free $\Bbbk {\bf P}$-module, we see that $\Bbbk {\bf G}\otimes_{\Bbbk {\bf P}}\Bbbk [{\bf P}/{\bf V}]\otimes_{\Bbbk {\bf L}}-$ is exact.

\end{proof}

\begin{Rem}
In the case of finite reductive groups, the Harish-Chandra restriction is right adjoint and also left adjoint to the Harish-Chandra induction  when $\op{char}\Bbbk\neq \op{char} \bar{\mathbb{F}}_q$. However this is not true in general for $\mathcal{R}^{{\bf G}}_{{\bf L}\subset {\bf P}}$ and ${}^*\mathcal{R}^{{\bf G}}_{{\bf L}\subset {\bf P}}$ in the case of infinite reductive groups. For example, we take ${\bf G}=SL_2(\bar{\mathbb{F}}_q)$, then
$$\op{Hom}_{\bf G}(\mathbb{M}(\op{tr}), \mathcal{R}^{{\bf G}}_{{\bf T}\subset {\bf B} }\op{tr})\ncong \op{Hom}_{\bf T}(\mathbb{M}(\op{tr}), \op{tr}).$$
Indeed we have $$\op{Hom}_{\bf G}(\mathbb{M}(\op{tr}), \mathcal{R}^{{\bf G}}_{{\bf T}\subset {\bf B} }\op{tr})=\op{End}_{\bf G}(\mathbb{M}(\op{tr})) \cong \op{Hom}_{\bf B}(\op{tr}, \mathbb{M}(\op{tr})) \cong \Bbbk.$$
However as $\Bbbk {\bf G}$ -module, $\mathbb{M}(\op{tr})$ has two composition factors: the infinite Steinberg module $\text{St}$ and the trivial module $\Bbbk_{\op{tr}}$. However, as  $\Bbbk {\bf T}$-module, each one has a simple quotient of trivial module, so we see that $\dim_{\Bbbk}\op{Hom}_{\bf T}(\mathbb{M}(\op{tr}), \op{tr})\geq 2.$

\end{Rem}

For $J\subset I$, there is a standard parabolic subgroup ${\bf P}_J$ with Levi decomposition ${\bf P}_J={\bf U}_J\rtimes {\bf L}_J$. Then we simply denote Harish-Chandra induction by $ \mathcal{R}^{J}:= \mathcal{R}^{ {\bf G}}_{ {\bf L}_J\subset {\bf P}_J}$ and Harish-Chandra restriction by $\mathcal{R}_{J}:= {}^*\mathcal{R}^{ {\bf G}}_{ {\bf L}_J\subset {\bf P}_J}$.
With this notation, the induced module $\mathbb{M}(\theta)$ is just $\mathcal{R}^{\emptyset}({\bf \Bbbk}_\theta)$ when $J$ is empty.

\medskip

Similar to the construction of $\mathscr{O}({\bf G})$, we  introduce the principal representation category $\mathscr{O}({\bf L}_J)$ of ${\bf L}_J$.
Let $\mathbb{M}_J(\theta)=\op{Ind}_{\Bbbk{\bf B}_J}^{\Bbbk{\bf L}_J}{\bf \Bbbk}_\theta$ be the standard induced module of ${\bf L}_J$, where ${\bf B}_J= {\bf U}_{w_J}\rtimes {\bf T}$ is a Borel subgroup of ${\bf L}_J$. Set $J(\theta)=J\cap I(\theta)$ and for $K\subset J(\theta)$, set
$$\mathbb{M}_J(\theta)_K=\displaystyle \sum_{w\in W_J}\Bbbk {\bf U}_{w_J}\dot{w}\eta(\theta)_K,$$
which is a submodule of $\mathbb{M}_J(\theta)$. For any $K\subset J(\theta)$, we define
$$E_J(\theta)_K=\mathbb{M}_J(\theta)_K / \mathbb{M}_J(\theta)_K',$$
where $\mathbb{M}_J(\theta)_K'$ is the sum of all $\mathbb{M}_J(\theta)_K$ with $J\subsetneq K\subset J(\theta)$. Analogous to the proof of Theorem \ref{EJ}(see \cite{CD3}), $E_J(\theta)_K$ is also irreducible and pairwise non-isomorphic. Let $\mathscr{O}({\bf L}_J)$ be the full subcategory of $\Bbbk {\bf L}_J$-modules containing objects whose subquotients are $E_J(\theta)_K$.

\medskip

Now for $J,K\subset I$, let $\{w_1,w_2,\cdots,w_m\}$ be a complete set of the representatives (with minimal length) of $(W_J,W_K)$-double cosets in $W$. We set $L_j=J\cap w_j(K)$. Then we have the following lemma.

\begin{Lem}\label{lemma1} Assume that $K\subset I(\theta)$. Then
$$\mathcal{R}_J(\Delta(\theta)_K) \cong \bigoplus_{j=1}^m \mathbb{M}_J(\theta^{w_j})_{L_j}$$
as $\Bbbk{\bf L}_J$-module.

\end{Lem}

\begin{proof}
Since $\Delta(\theta)_K=\sum_{w\in X_K}\Bbbk {\bf U}\dot{w}\eta(\theta)_K$,
where $$X_K =\{x\in W\mid x~\text{has~minimal~length~in}~xW_K\},$$
we have
$$
\aligned\mathcal{R}_J(\Delta(\theta)_K) &\ = \Bbbk [{\bf U}_J\backslash {\bf G}] \otimes_{\Bbbk {\bf G}} \Delta(\theta)_K \\
&\ \cong\bigoplus_{j=1}^m \Bbbk{\bf L}_J({\bf U}_J\otimes w_j \eta(\theta)_K ),
\endaligned
$$
where $\Bbbk{\bf L}_J({\bf U}_J\otimes w_j \eta(\theta)_K )=:\Lambda_j $ is the $\Bbbk {\bf L}_J$-module generated by the element ${\bf U}_J\otimes w_j \eta(\theta)_K $.

By the setting  of $w_j$ and $L_j=J\cap w_j(K)$, we have $W_{L_j}=W_J\cap \ ^{w_j}W_K $ and $\Phi_{L_j}^{+}= \Phi_{J}^{+}\cap w_j(\Phi_{K}^{+})$ by  Theorem 2.6 in \cite{CE}.
So it is not difficult to verify that
$$f: \ \mathbb{M}_J(\theta^{w_j})_{L_j} \longrightarrow \Lambda_j,\ \ \eta(\theta^{w_j})_{L_j}\longrightarrow {\bf U}_J\otimes w_j \eta(\theta)_K$$
is a $\Bbbk {\bf L}_J$-module homomorphism. Moreover it is also an isomorphism which completes the proof.
\end{proof}

On the other hand we consider $\mathcal{R}^J(\mathbb{M}_J(\theta)_K)$ and the following lemma is easy to check.
\begin{Lem} \label{lemma2}For $J\subset I$ and $K\subset J(\theta)$, one has
$$\mathcal{R}^J(\mathbb{M}_J(\theta)_K)\cong \Delta(\theta)_K$$
as $\Bbbk {\bf G}$-modules.
\end{Lem}

\begin{proof} Both modules on left- and right-hand sides can be regarded as $\Bbbk {\bf G}$-modules generated by the element $\eta(\theta)_K$.
\end{proof}

In particular, using Lemma \ref{lemma1} and Lemma \ref{lemma2},  we get
$$\mathcal{R}^{J}(\mathbb{M}_J(\theta))\cong \mathbb{M}(\theta)$$ as $\Bbbk {\bf G}$-modules and $$\mathcal{R}_{J}(\mathbb{M}(\theta))=\Bbbk [{\bf U}_J\backslash {\bf G}] \otimes_{\Bbbk {\bf G}}\mathbb{M}(\theta)\cong \bigoplus_{j=1}^k \mathbb{M}_J(\theta^{v_j})$$
as $\Bbbk{\bf L}_J$-module, where $\{v_1,v_2,\cdots,v_k\}$ is a complete set of the representatives (with minimal length) of the right cosets of $W_J$ in $W$. Since $\mathcal{R}^{J}$ and $\mathcal{R}_{J}$ are both exact functors, we see that
$$\mathcal{R}^{J}: \  \mathscr{O}({\bf L}_J)\longrightarrow \mathscr{O}({\bf G})\ \ \text{and} \ \ \mathcal{R}_{J}: \  \mathscr{O}({\bf G}) \longrightarrow \mathscr{O}({\bf L}_J).$$
are two well-defined functors.

\section{Alvis-Curits Duality}

By the results in the last section,
$$\mathcal{R}^{J}: \  \mathscr{O}({\bf L}_J)\longrightarrow \mathscr{O}({\bf G})\ \ \text{and} \ \ \mathcal{R}_{J}: \  \mathscr{O}({\bf G}) \longrightarrow \mathscr{O}({\bf L}_J).$$
are exact functors. Thus we can define them on the Grothendieck group with the same notation by
$$ \mathcal{R}^{J}: \ \ K_0(\mathscr{O}({\bf L}_J))\longrightarrow K_0(\mathscr{O}({\bf G})),\ \ \ [N]\longrightarrow [\mathcal{R}^{J} N],$$
$$ \mathcal{R}_{J}: \ \ K_0(\mathscr{O}({\bf G}))\longrightarrow K_0(\mathscr{O}({\bf L}_J)),\ \ \ [M]\longrightarrow [\mathcal{R}_{J} M].$$

The ordinary Alvis-Curits duality is a duality operation on the characters of finite reductive groups; see \cite{DM} for more details. Analogously we can define
$\mathbb{D}_{\bf G}: K_0(\mathscr{O}({\bf G}))\longrightarrow K_0(\mathscr{O}({\bf G}))$
by $$\mathbb{D}_{\bf G}= \sum_{J\subset I} (-1)^{|J|}\ \mathcal{R}^J  \  \mathcal{R}_J$$
and call $\mathbb{D}_{\bf G}$ the Alvis-Curits duality of infinite type on $\mathscr{O}({\bf G})$.
In the following we try to understand how $\mathbb{D}_{\bf G}$ operates on the elements in $K_0(\mathscr{O}({\bf G}))$ and then we will  show that $\mathbb{D}_{\bf G}$ is really a duality.

\medskip

By Proposition \ref{two basis}, $\{[\Delta(\theta)_K] \mid \theta\in \widehat{\bf T}, K\subset I(\theta)\}$ is a basis of $K_0(\mathscr{O}({\bf G}))$, so we need to  consider $\mathbb{D}_{\bf G}([\Delta(\theta)_K])$ for a fixed $\theta\in \widehat{\bf T}$ and $K\subset I(\theta)$. For each $J\subset I$, denote by $D_{J}$ a complete set of the representatives (with minimal length) of $(W_J,W_K)$-double cosets in $W$. For each $^J{w_j}\in D_{J}$, set $^J{L_j}=J\cap \ ^J{w_j}(K)$. Then using Lemma \ref{lemma1} and Lemma \ref{lemma2}, we have
$$\mathcal{R}^J \mathcal{R}_J(\Delta(\theta)_K) \cong \bigoplus_{^J{w_j}\in D_J} \Delta(\theta^{^J{w_j}})_{^J{L_j}}.$$
Therefore we see that
$$\mathbb{D}_{\bf G}([\Delta(\theta)_K])=\sum_{J\subset I} (-1)^{|J|} \sum_{^J{w_j}\in D_J}[\Delta(\theta^{^J{w_j}})_{^J{L_j}}].$$
For $w\in W$, we set $\mathscr{H}(w)=\{i\in I\mid s_iw> w\}$. So the previous formula can also be written as
$$\mathbb{D}_{\bf G}([\Delta(\theta)_K])= \sum_{w\in X_K} \sum_{J\subset \mathscr{H}(w)}(-1)^{|J|}
[\Delta(\theta^{w})_{J\cap w(K)}],$$
where $X_K =\{x\in W\mid x~\text{has~minimal~length~in}~xW_K\}$ as before.

We claim that $w(K)\subseteq \mathscr{H}(w)$ for $w\in X_K$. Indeed, assume $j\in w(K)$, then there exists $i \in K$ such that $w(\alpha_i)=\alpha_j$ which implies $ws_i>w$. On the other hand, $s_jw(\alpha_i)=-\alpha_j$, and thus $s_jws_i<s_jw$. Therefore we have $ws_i=s_jw>w$ and $w(K)\subseteq \mathscr{H}(w)$. Moreover $w(K)= \mathscr{H}(w)$ if and only if $w=w_0w_K$.

It is easy to check that when $w(K)\subsetneq \mathscr{H}(w)$, then
$$\sum_{J\subset \mathscr{H}(w)}(-1)^{|J|}
[\Delta(\theta^{w})_{J\cap w(K)}]=0.$$
We set $\sigma(K)=\mathscr{H}(w_0w_K)$ for $K\subset I(\theta)$.
Then $$\sigma(K)=\{i\in I\mid s_i=w_0s_jw_0 \ \text{for some} \ j\in K \}$$ which is a subset of $\sigma(I(\theta))=I(\theta^{w_0})$.
Noting that $\theta^{w_0}=\theta^{w_0w_K}$ for $K\subset I(\theta)$, we have proved  the following lemma.
\begin{Lem} \label{lemma3}For $\theta\in \widehat{\bf T}$ and $K\subset I(\theta)$, we have
$$\mathbb{D}_{\bf G}([\Delta(\theta)_K])=\sum_{J\subset \sigma(K)}(-1)^{|J|}[\Delta(\theta^{w_0})_J].$$
\end{Lem}

If we set $\theta$ trivial and $K=I$, we see that $\Delta(\op{tr})_I=\op{St}$, where $\op{St}$ is the infinite dimensional Steinberg module (see \cite{Y}). By Lemma \ref{lemma3}, we have
$$\mathbb{D}_{\bf G}([\op{St}])=\sum_{J\subset I}(-1)^{|J|}[\Delta(\op{tr})_J].$$
The right-hand side of the previous equation is just $[{\bf \Bbbk}_{\op{tr}}]$ by Proposition \ref{two basis}. Thus $\mathbb{D}_{\bf G}([\op{St}])=[{\bf \Bbbk}_{\op{tr}}]$.
It suggests us to consider the formula of $\mathbb{D}_{\bf G}[E(\theta)_J]$ for each $J\subset I(\theta)$.

Recall the notation
$$\sigma(K)=\{i\in I\mid s_i=w_0s_jw_0 \ \text{for some} \ j\in K \}$$ for a subset $K$ of $I$.
Now for a fixed $\theta \in \widehat{\bf T} $, we introduce a function $\mathbb{D}_{\theta}$ from  the set of subsets of  $I(\theta)$ to the set of subsets of $I(\theta^{w_0})$ . For $J\subset I(\theta)$,  we set $$\mathbb{D}_{\theta}(J)=\sigma(I(\theta)) \backslash \sigma(J)= I(\theta^{w_0}) \backslash \sigma(J).$$ Then it is easy to check that $\mathbb{D}_{\theta^{w_0}} (\mathbb{D}_{\theta}(J))=J$ with $J\subset I(\theta)$.
\begin{Thm} \label{dual of basis} For $\theta\in \widehat{\bf T}$ and $J\subset I(\theta)$, we have
$$\mathbb{D}_{\bf G}([E(\theta)_J])=[E(\theta^{w_0})_{\mathbb{D}_{\theta}(J)}].$$
\end{Thm}
\begin{proof}We calculate it directly. By Proposition \ref{two basis}, we have
$$\mathbb{D}_{\bf G}([E(\theta)_J])= \sum_{J\subset K \subset I(\theta)}(-1)^{|K|-|J|}\mathbb{D}_{\bf G}([\Delta(\theta)_K])$$
Using Lemma \ref{lemma3}, the above formula is equal to
$$\sum_{J\subset K \subset I(\theta)}(-1)^{|K|-|J|} \sum_{L\subset \sigma(K)} (-1)^{|L|}[\Delta(\theta^{w_0})_L],$$
which is also equal to
$$\sum_{L\subset \sigma(I(\theta))}(-1)^{|L|-|J|}(\sum_{\sigma(K)\supset L\cup \sigma(J)}(-1)^{|K|} ) [\Delta(\theta^{w_0})_L].$$

When $\sigma({I(\theta)})=I(\theta^{w_0})\supsetneq L\cup \sigma(J)$, we have $$\displaystyle \sum_{\sigma(K)\supset L\cup \sigma(J)}(-1)^{|K|}=0.$$ When $I(\theta^{w_0})= L\cup \sigma(J)$ which implies that $L\supset I(\theta^{w_0})\backslash \sigma(J)=\mathbb{D}_{\theta}(J)$, thus in this case we have $$\displaystyle\sum_{\sigma(K)\supset L\cup \sigma(J)}(-1)^{|K|}=(-1)^{|I(\theta^{w_0})|} .$$
Therefore we get
$$
\aligned
\mathbb{D}_{\bf G}([E(\theta)_J])&\ = \sum_{\mathbb{D}_{\theta}(J) \subset L \subset I(\theta^{w_0})}
(-1)^{|L|-|J|+|I(\theta^{w_0})|} [\Delta(\theta^{w_0})_L]\\
&\ = \sum_{\mathbb{D}_{\theta}(J) \subset L \subset {I(\theta^{w_0})}}
(-1)^{|L|-|\mathbb{D}_{\theta}(J)|} [\Delta(\theta^{w_0})_L]= [E(\theta^{w_0})_{\mathbb{D}_{\theta}(J)}]
\endaligned
$$
by  Proposition \ref{two basis}. The theorem is proved.
\end{proof}

Since $\{[E(\theta)_J] \mid \theta\in \widehat{\bf T}, J\subset I(\theta)\}$ is a basis of $K_0(\mathscr{O}({\bf G}))$, we get the following corollary by Theorem \ref{dual of basis}.
\begin{Cor} $\mathbb{D}_{\bf G}\circ \mathbb{D}_{\bf G}$ is the identity functor on $K_0(\mathscr{O}({\bf G}))$. So $\mathbb{D}_{\bf G}$ is called the Alvis-Curits duality of infinite type.
\end{Cor}

\begin{Cor} For $\theta\in \widehat{\bf T}$ and $J\subset I(\theta)$, we have
$$\mathbb{D}_{\bf G}([\Delta(\theta)_J])=[\nabla(\theta^{w_0})_{\mathbb{D}_{\theta}(J)}].$$
\end{Cor}
\begin{proof} By Proposition \ref{two basis} and Theorem \ref{dual of basis}, we have
$$
\aligned
\mathbb{D}_{\bf G}([\Delta(\theta)_J])= \sum_{K \supset J} \mathbb{D}_{\bf G}([E(\theta)_K])=\sum_{K \supset J} [E(\theta^{w_0})_{\mathbb{D}_{\theta}(K)}]
\endaligned
$$
Note that $\mathbb{D}_{\theta}(K)\subset \mathbb{D}_{\theta}(J)$ if and only if  $K \supset J$. Thus we have
$$
\aligned
\mathbb{D}_{\bf G}([\Delta(\theta)_J])=\sum_{\mathbb{D}_{\theta}(K) \subset \mathbb{D}_{\theta}(J)} [E(\theta^{w_0})_{\mathbb{D}_{\theta}(K)}]=[\nabla(\theta^{w_0})_{\mathbb{D}_{\theta}(J)}]
\endaligned
$$
also by Proposition \ref{two basis}.
\end{proof}

\bigskip

\bibliographystyle{amsplain}

\end{document}